\documentclass[a4paper,10pt]{article}
\makeatletter
\usepackage[T1]{fontenc}
\usepackage[latin1]{inputenc}
\def\@tempa{:}
\if\@currdir\@tempa
\usepackage[applemac]{inputenc}
\fi
\DeclareRobustCommand{\qed}{%
  \ifmmode 
  \else \leavevmode\unskip\penalty9999 \hbox{}\nobreak\hfill
  \fi
  \quad\hbox{\qedsymbol}}
\newcommand{\emptysetpenbox}{\leavevmode
  \hbox to.77778em{%
  \hfil\vrule
  \vbox to.675em{\hrule width.6em\vfil\hrule}%
  \vrule\hfil}}
\newcommand{\qedsymbol}{\emptysetpenbox}
\newenvironment{proof}[1][\proofname]{\par
  \normalfont
  \topsep6\p@\@plus6\p@ \trivlist
  \item[\hskip\labelsep\itshape
    #1\@addpunct{.}]\ignorespaces
}{%
  \qed\endtrivlist
}

\newcommand{\proofname}{Proof}
\newenvironment{exemple*}[1][\examplename]{\par
  \normalfont
  \topsep6\p@\@plus6\p@ \trivlist
  \item[\hskip\labelsep\itshape\bfseries
     #1\@addpunct{.}]\ignorespaces\footnotesize
}{%
  \hfill $\clubsuit$\endtrivlist
}
\makeatother
\usepackage{enumerate}
\usepackage[centertags]{amsmath}
\usepackage[dvips]{graphics} 
\usepackage[T1]{fontenc}
\usepackage{amssymb}
\usepackage{pifont}
\usepackage{rotating}
\rotdriver{dvips}
\usepackage{theorem}
\usepackage{verbatim}
\author{P. Moyal\\
Laboratoire de Math\'ematiques Appliqu\'ees de Compi\`egne\\
Universit\'e de Technologie de Compiègne\\
D\'epartement G\'enie Informatique\\
Centre de Recherches de Royallieu\\
BP 20 529\\
60 205 Compiègne Cedex\\
France}
\title{Construction of a stationary FIFO queue with impatient customers}

\newcommand\car{{\mathbf 1}}
\newcommand\N{{\mathbb N}}

\renewcommand\P{{\mathbf P}} 
\newcommand{\B}{{\mathfrak B}}
\newcommand{\A}{{\mathfrak  A}}

\newcommand{\R}{{\mathbb R}}
\newcommand{\F}{{\mathcal F}}
\newcommand{\W}{{\mathcal W}}

\newcommand\prp[1]{{\mathbf P}^0\left[#1\right]}

\newcommand\pp{{\mathbf P^0}}
\newcommand\qp{{\tilde{\mathbf P}^0}}

\newcommand{\pae}[1]{\mbox{$\lfloor \kern-1pt #1 \kern-1pt \rfloor$}}
\newcommand{\paep}[1]{\mbox{$\lceil \kern-1pt #1 \kern-1pt \rceil$}}

\def\N{{\mathbb N}}
\def\Z{{\mathbb Z}}

\def\P{{\mathbf P}}

\newcommand\procz[1]{\left(#1_t\right)_{t \in \R}}

\newcommand\suite[1]{\left\{#1_n\right\}_{n \in \N}}
\newcommand\suiten[1]{\left\{#1\right\}_{n \in \N}}
\newcommand\suitez[1]{\left\{#1_n\right\}_{n \in \Z}}

{\theoremstyle{break} 
\theorembodyfont{\bfseries}
\newtheorem{theorem}{Theorem}[section]

\newtheorem{lemma}[theorem]{Lemma}

}
{\theorembodyfont{\upshape}
}
{\theorembodyfont{\upshape\small}

}
\date{}
\begin{document}
\maketitle

\begin{abstract}
In this paper, we study the stability of queues with impatient customers. 
Under general stationary ergodic assumptions, we first provide some conditions for such a queue to be regenerative (i.e. to empty a.s. an infinite number of times). In the particular case of a single server operating in First in, First out, we prove the existence (in some cases, on an enlarged 
probability space) 
of a stationary workload. This is done by studying stochastic recursions under the Palm settings, and by stochastic comparison of 
stochastic recursions. 
\end{abstract}
\emph{keywords}: Stochastic recursions; Stationary solutions; Queues with impatience; Renovative events; Enlargement of probability space.\\\\ 
\emph{subject class (2000)}: Primary : 60F17, Secondary :  60K25 ; 60B12.
\section{Introduction}
%

In this paper, we address the question of stability for queueing systems with impatient customers: the customers agree to wait for their 
service only during a limited period of time. They are discarded from the system provided that their patience ends before they could reach the service booth. 
Such models are particularly adequate to describe operating systems under sharp delay requirements: 
multimedia and time sensitive telecommunication and computer networks, on-line audio/video traffic flows, call centres or supply chains. 

We first give conditions of regenerativity, i.e., for the state zero to be recurrent for the congestion process of the system. Then, we construct explicitly a stationary state for these systems in the particular case of a single server obeying the FIFO (First In, First Out) discipline. 
To that end, we study a stochastically recursive sequence representing the workload seen by an arriving customer. This sequence 
 and its dynamics have been thoroughly studied in the GI/GI/1 case in \cite{HebBac81} and \cite{MR86d:60103}. In the G/G/1 context, 
the workload sequence is driven by a non-monotonic recursive equation (eq. (\ref{eq:recurstatFIFO})), and hence a construction of Loynes's type, using a backwards recurrence scheme, is not possible. We thus use more recent and sophisticated techniques to construct a stationary workload: (i) Borovkov's theory of renovating events (see \cite{MR52:12118}, \cite{Foss92}, \cite{BacBre02}) provides a sufficient condition for the existence and uniqueness of a finite stationary workload. Under this condition, we can thus construct a stationary loss probability $\pi$, and provide bounds for $\pi$ (eq. (\ref{eq:encadrepiB})). (ii) We prove in whole generality the existence of a stationary workload on the enriched probability space $\Omega\times\R+$ (where $\Omega$ is the Palm probability space of reference) using Anantharam and Konstantopoulos' construction (see \cite{Anan97}, \cite{Anan99}), which is based on tightness techniques. 

In both cases, 
we use the fact that the workload sequence is strongly dominated by another one, that is driven by a monotonic recursive equation (eq. (\ref{eq:recur})).  
Then the coupling of the dominating sequence with a unique stationary state (which is proven by Loynes' scheme) allows us to construct the  stationary state of the dominated sequence. 

We address as well the case, where the customers are impatient until the end of their service: they are lost provided their deadline is reached before a server could \emph{complete} their service. This case happens to be more simple, in that the workload sequence is driven by a 
monotonic and continuous recursive random mapping. Then the stability question can be handled by Loynes' scheme.

This paper is organized as follows. 
In section \ref{sec:PRELIM} we make precise our basic assumptions, and solve eq. (\ref{eq:recur}) in the stationary ergodic framework,  
 a result that will be used in the sequel. We present the queue with impatience until the beginning of service in section \ref{sec:IB}. 
In section \ref{sec:regenerIB}, we provide conditions for the regenerativity of this system. In section \ref{sec:FIFOIB}, we construct a 
stationary workload 
in the FIFO case: we provide a sufficient condition for the existence and uniqueness in \ref{subsubsec:borov}, and prove the existence of the stationary workload on an enriched probability space in \ref{subsubsec:anan}. Finally, in section \ref{sec:FIFOIE} we study the case of impatience until the end of service. 

\section{Preliminaries}
\label{sec:PRELIM}

Consider a probability space $\left(\Omega,\F,\pp\right)$, embedded with the 
measurable bijective flow $\theta$ (denote $\theta^{-1}$, its measurable 
inverse). Suppose that $\pp$ is stationary and ergodic under $\theta$, 
\emph{i.e.}  
for all $\A \in \F$, $\prp{\theta^{-1}\A}=\prp{\A}$ and all $\mathfrak A$ that is $\theta$-invariant (i.e. such that $\theta\A=\A$) is of  probability $0$ or $1$. Note that according to these axioms, all $\theta$-contracting event (such that $\prp{\A^c\cap\theta^{-1}\A}=0$) is 
of probability $0$ or $1$.  Let $\Z$, $\N$ and $\N^*$ denote the sets of integers, of non-negative integers and of positive integers, respectively. We denote for all $n \in \N$, 
$\theta^n=\theta\circ\theta\circ...\circ\theta\,\,\mbox{ and }\,\,\theta^{-n}=\theta^{-1}\circ\theta^{-1}\circ...\circ\theta^{-1},$ and say that two 
sequences of r.v. $\suite{X}$ and $\suite{Y}$ couple when there exists a $\pp$-a.s. finite rank $N$ such that they coincide for all $n \ge N$. 
We say that there is strong backwards coupling between $\suite{X}$ and $\suiten{Y\circ\theta^n}$ provided that for some $\pp$-a.s. finite $\tau$, 
$X_n\circ\theta^{-n}=Y$ for any $n \ge \tau$. We denote for any $x,y \in \R$, $x\vee y=\max(x,y)$, $x\wedge y=\min(x,y)$ and $x^+=x\vee 0.$ 

Let $\alpha$ and $\beta$ be two integrable $\R+$-valued r.v. such that $\prp{\beta>0}>0$ and denote for all $n \in \Z$, $\alpha_n=\alpha\circ\theta^n$ and $\beta_n=\beta\circ\theta^n$. 
Let $Z$ be an a.s. finite $\R+$-valued r.v., and consider the following stochastically recursive sequence. 
\[\left\{\begin{array}{ll}
Y^Z_0&=Z,\\
Y^Z_{n+1}&=\left[\max\left\{Y^Z_n,\alpha_n\right\}-\beta_n\right]^+\mbox{ for all }n \in \N.
\end{array}\right.\] 
Then, $\suite{Y^{Y_{\alpha,\beta}}}$ is a stationary version of this sequence provided that the r.v. $Y_{\alpha,\beta}$ is a solution to  
\begin{equation}
\label{eq:recur}
Y_{\alpha,\beta}\circ\theta=\left[Y_{\alpha,\beta}\vee\alpha-\beta\right]^+.
\end{equation} 
We have the following result. 
\begin{lemma}
\label{lemma:solrecur} 
There exists a unique $\pp$-a.s. finite solution $Y_{\alpha,\beta}$ of (\ref{eq:recur}), 
given by
\begin{equation}
  \label{eq:limMn}
  Y_{\alpha,\beta}:=\left[\sup_{j \in \N^*} \left(\alpha_{-j}-\sum_{i=1}^{j}\beta_{-i}\right)\right]^+.
\end{equation}
Moreover, for all $\pp$-a.s. finite and nonnegative r.v. $Z$, the sequence $\suite{Y^Z}$ couples with 
$\left\{Y_{\alpha,\beta}\circ\theta^n\right\}_{n \in \N}$, and there exists $\pp$-a.s. an infinity of indices such that $Y^Z_n=0$ if and only if 
\begin{equation}
\label{eq:regener}
\prp{Y_{\alpha,\beta}=0}>0.
\end{equation}
 
\end{lemma}

\begin{proof}
Equation (\ref{eq:recur}) can be handled by Loynes's construction 
(see \cite{Loynes62}, \cite{BacBre02}) since the mapping 
$x\mapsto\left[x\vee\alpha-\beta\right]^+$ 
is $\pp$-a.s. continuous and non-decreasing. Hence $Y_{\alpha,\beta}$ classically reads as the $\pp$-a.s. limit of Loynes's sequence, defined by 
$\suiten{Y^0_n\circ\theta^{-n}}$. It is routine to check from Birkhoff's ergodic theorem (and the fact that $\beta$ is not identically zero) 
that $Y_{\alpha,\beta}$ is $\pp$-a.s. finite. 
The coupling property follows from the fact that for all non-negative r.v. $Z$ that is $\pp$-a.s finite (and in particular, for $Z=Y_{\alpha,\beta}$), 
\begin{equation*}
\left\{Y_n^Z\neq Y_n^0 \mbox{ for all }n \in \N\right\}= \left\{Y_n^Z=Z-\sum_{i=0}^{n-1}\beta_i>0\mbox{ for all }n \in \N\right\},
\end{equation*} 
which is of probability 0 from Birkhoff's theorem. 
The last statement is a classical consequence of this coupling property under ergodic assumptions. 
\end{proof} 

\section{The model}
\label{sec:IB}
Let us consider a queue with 
impatient customers until the beginning of service G/G/s/s+G(b) (according to Barrer's notation, see \cite{MR19:779g}): 
 on the 
probability space $\left(\Omega,\F, \mathbf P\right)$, furnished with the measurable bijective flow $\procz{\theta}$ under which  
$\mathbf P$ is stationary and ergodic,  
consider the $\theta_t$-compatible point process $N$, whose points $\suitez{T}$ 
represent the 
arrivals of the customers $\suitez{C}$, with the convention that $T_0$ is the 
last arrival before time $t=0$. 
The interarrivals are denoted for all $n \in \Z$ by 
$\xi_n=T_{n+1}-T_n$.   The process $N$ is 
marked by the sequence $\suitez{\sigma}$ of non-negative r.v. representing the service durations requested by the customers. The queueing system 
has $s$ non-idling servers and is of infinite capacity. The customers are impatient until the beginning of their service, in that they leave the system if they do not reach the service 
booth before a given deadline. In other words, customer $C_n$ agrees to wait in line for a given period of time, say $D_n$ (his initial \emph{patience}) and if there is no server available during this period, he leaves the system forever at time $T_n+D_n$.  
$\suitez{D}$ is a sequence of non-negative marks of $\procz{N}$. We denote for all $t \in \R$, $\mathcal X_t$ the number of customers in 
the system (or \emph{congestion}) at $t$.   
The servers follow a non-preemptive service discipline. We therefore consider that a customer can 
not be eliminated 
anymore as soon as he enters the service booth, even if his deadline is reached during his service.   

Let us denote $\left(\Omega,\F,\pp,\theta\right)$, the Palm probability space 
of $N(\sigma,D)$, where $\theta:=\theta_{T_1}$ is the associated bijective discrete  
flow. Then, $\pp$ is stationary and ergodic under $\theta$, and the sequence $\{\xi_n,\sigma_n,D_n\}_{n\in\Z}$ is stationary in that for 
all $n$, 
$$\xi_n=\xi\circ\theta^n\,\mbox{, }\, \sigma_n=\sigma\circ\theta^n\,\mbox{ and }\,D_n=D\circ\theta^n, $$ where  
$\xi:=\xi_0$, $\sigma:=\sigma_0$ and $D:=D_0$. It is furthermore assumed that $\xi$, $\sigma$ and $D$ are $\pp$-integrable.

\section{Regenerativity}
\label{sec:regenerIB}
According to the assumptions made above, the total sojourn time of customer 
$C_n$ does not 
exceed $D_n+\sigma_n$, i.e. the sum of his initial patience and the time necessary for his service. 
On the other hand, it is at least equal to $\sigma_n\wedge D_n$, 
i.e. the time needed for him to be lost, or immediately served. Hence, 
provided  that $C_n$ entered the system before $t$ ($T_n \le t$) and even 
though he already left the system before $t$, his 
\emph{remaining maximal sojourn time} at $t$ 
(i.e. the remaining time before his  
latest possible departure time, if not already reached) is given by $\left[\sigma_n+D_n-(t-T_n)\right]^+,$ whereas his \emph{remaining minimal sojourn time} at $t$ 
(i.e. the remaining time before his earliest possible departure time, if not already reached) 
is given by $\left[\sigma_n\wedge D_n-(t-T_n)\right]^+.$  
Hence the \emph{largest remaining maximal sojourn time} 
(LRMST for short) at $t$ among all the customers entered before $t$ is given by
$$
\mathcal L_t:=\max_{n=1}^{N_t}\left[\sigma_n+D_n-(t-T_n)\right]^+$$
and the largest remaining minimal sojourn time 
(LRmST for short) at $t$, by 
$$\mathcal M_t:=\max_{n=1}^{N_t}\left[\sigma_n\wedge D_n-(t-T_n)\right]^+.$$ 
The two processes $\procz{\mathcal L}$ and $\procz{\mathcal{M}}$ evolve according to the following dynamics: they decrease 
at unit rate between arrival times, and equal the initial maximal (resp. minimal) 
sojourn time of $C_n$ at his arrival time $T_n$, provided that it is larger 
than the value of the process just before $T_n$. 
In other words, for all $n \in \Z$ and all $t \in \left[T_n,T_{n+1}\right)$ 
\begin{equation*}
\mathcal L_t=\left[\max\left\{\mathcal L_{T_n-},\sigma_n+D_n\right\}-\left(t-T_n\right)\right]^+,
\end{equation*} 
\begin{equation*}
\mathcal M_t=\left[\max\left\{\mathcal M_{T_n-},\sigma_n\wedge D_n\right\}-\left(t-T_n\right)\right]^+.
\end{equation*} 
Define for all finite nonnegative r.v. $Y$ and $Z$ and all $n \in \N$, 
$L_n^Y:=\mathcal L_{T_n-}$ and $M_n^Z:=\mathcal {M}_{T_n-}$, 
the LRMST (resp. LRmST) just before the arrival of customer $C_n$, provided 
that $\mathcal L_{T_0-}=Y$ (resp. $\mathcal M_{T_0-}=Z$). The processes $\procz{\mathcal L}$ and 
$\procz{\mathcal {M}}$ have rcll paths, hence we have the following recursive 
equations.
\begin{equation}
\label{eq:recurL}
L^Z_{n+1}=\left[\max\left\{L^Z_n,\sigma_n+D_n\right\}-\xi_n\right]^+,
\end{equation}
$$
M^Z_{n+1}=\left[\max\left\{M^Z_n,\sigma_n\wedge D_n\right\}-\xi_n\right]^+.$$ 
On $\left(\Omega,\pp\right)$ the two latter equations are of type 
(\ref{eq:recur}), hence 
Lemma \ref{lemma:solrecur} implies that for any $Y$ and $Z$, $\suite{L^Y}$ and $\suite{M^Z}$ couple  
respectively with $\suiten{Y_{\sigma+D,\xi}\circ\theta^n}$ and $\suiten{Y_{\sigma\wedge D,\xi}\circ\theta^n}$, where  
\begin{equation}
\label{eq:defL}
Y_{\sigma+D,\xi}=\left[\sup_{j \in \N^*} \left(\sigma_{-j}+D_{-j}-\sum_{i=1}^{j}\xi_{-i}\right)\right]^+,
\end{equation} 
\begin{equation}
\label{eq:defM}
Y_{\sigma\wedge D,\xi}=\left[\sup_{j \in \N^*} \left(\sigma_{-j}\wedge D_{-j}-\sum_{i=1}^{j}\xi_{-i}\right)\right]^+.
\end{equation} 

In particular, for any initial conditions $L$ and $M$, there are $\P$-a.s. an infinity of indices such that 
$L^L_n=0$ if and only if $\prp{Y_{\sigma+D,\xi}=0}>0$, and an 
infinity of indices such that $M^M_n=0$ if and only if 
$\prp{Y_{\sigma\wedge D,\xi}=0}>0$. Remarking now that for all initial 
conditions and all $t \in \R$,
$\left\{\mathcal L_t=0\right\}\subseteq\left\{\mathcal X_t=0\right\}\subseteq\left\{\mathcal M_t=0\right\},$ 
we obtain  
\begin{theorem}
\label{thm:regenerGGb}
The G/G/s/s+G(b) queue is regenerative (i.e. it empties $\pp$-a.s. an infinite number of times) if 
\begin{equation}
\label{eq:condsuffGGd}
\prp{Y_{\sigma+D,\xi}=0}>0
\end{equation}
and only if
\begin{equation}
\label{eq:condnecGGd}
\prp{Y_{\sigma\wedge D,\xi}=0}>0,
\end{equation}
where $Y_{\sigma+D,\xi}$ and $Y_{\sigma\wedge D,\xi}$ are defined respectively by (\ref{eq:defL}) and (\ref{eq:defM}).
\end{theorem}

\section{The FIFO case}
\label{sec:FIFOIB}
Let us now consider the special case, where there is one server obeying the FIFO (\emph{First in, first out}) 
discipline.  Denote for all 
$t \in \R$, $\W_t$ the workload submitted to the server at time $t$, i.e. the 
quantity of work he still has to achieve at this time, in time unit. 
The process $\procz{\W}$ 
has rcll 
paths, and we define for all $n$, $W_n=\W_{T_n^-}$. Its value at $t$ equals 
the work brought by the customers arrived up to $t$, and who will 
eventually be served, since the other ones won't ever reach the server. 
Under the FIFO discipline, the served customers are those who find a 
workload less than their patience upon arrival. In-between arrival times, 
$\procz{\W}$ decreases at unit rate. Hence, for all $n \in \Z$ and $t\in \left[T_n,T_{n+1}\right)$
$$\W_t=\left[\W_{T_n-}+\sigma_n\car_{\left\{\W_{T_n-} \le D_n\right\}}-\left(t-T_n\right)\right]^+,$$   
whereas the workload sequence is driven by the recursive equation
\begin{equation}
\label{eq:recFIFO}
W_{n+1}=\left[W_n+\sigma_n\car_{\left\{W_n \le D_n\right\}}-\xi_n\right]^+.
\end{equation}

For all $n \in \N$ and all finite non-negative r.v. Z, let $W^Z_n$ be the workload seen by $C_n$ upon arrival, 
provided that $W^Z_0=Z$. 
In addition to the previous result of regenerativity 
(Theorem \ref{thm:regenerGGb}), we investigate in this particular case the 
existence and uniqueness of a stationary version for  
the stochastic recursion (\ref{eq:recFIFO}), i.e., of a finite r.v. $W$ such that $W^W_n=W\circ\theta^n$, $n\in\N$, which implies that
\begin{equation}
\label{eq:recurstatFIFO}
W\circ\theta=\left[W+\sigma\car_{\left\{W\le D\right\}}-\xi\right]^+.
\end{equation} 
The recursive equation (\ref{eq:recurstatFIFO}) is not monotonic in the state variable, in which case a construction 
of Loynes's type becomes fruitless. In the following sections, we propose two methods to circumvent this difficulty. 
\subsection{Sufficient condition}
\label{subsubsec:borov}
In Theorem \ref{thm:borov}, we use Borovkov's theory of renovating events to provide a sufficient condition for the existence of a solution to 
(\ref{eq:recurstatFIFO}), and for strong backwards coupling to occur.  
\begin{theorem}
\label{thm:borov}
If (\ref{eq:condsuffGGd}) holds, then (\ref{eq:recurstatFIFO}) admits a unique finite solution $W$, which is such that $Y_{\sigma\wedge D,\xi}\le W \le Y_{\sigma+D,\xi}$, 
$\pp$-a.s.. Moreover, for any initial condition $Z$ such that $Z\le Y_{\sigma+D,\xi}$, $\pp$-a.s., there is strong backwards coupling for $\suite{W^Z}$ with $\suiten{W\circ\theta^n}$. 
\end{theorem}
\begin{proof}
\emph{Existence}. 
Let us first remark that for any $x\in\R^+$,  
\begin{multline}
\label{eq:ineq}
x+\sigma\car_{x\le D}=(x+\sigma)\car_{x\le D}+x\car_{D<x\le D+\sigma}+x\car_{D+\sigma<x}\\
\le (D+\sigma)\car_{x\le D}+(D+\sigma)\car_{D<x\le D+\sigma}+x\car_{D+\sigma<x}\\=x\vee(D+\sigma),\mbox{ $\pp$-a.s.}.
\end{multline}
Hence for any $x\le y$ , 
$$\left[x+\sigma\car_{x\le D}-\xi\right]^+\le\left[x\vee(D+\sigma)-\xi\right]^+\le \left[y\vee(D+\sigma)-\xi\right]^+,\pp-\mbox{a.s.},$$
and a straightforward induction shows that $Z\le Y_{\sigma+D,\xi}$, $\pp$-a.s. implies that for all $n \ge 0$,  
$W^Z_n\le L^{Y_{\sigma+D,\xi}}_n = Y_{\sigma+D,\xi}\circ\theta^n$, $\pp$-a.s.. 
Therefore, denoting $\mathfrak A_n,$ the event $\left\{Y_{\sigma+D,\xi}\circ\theta^n=0\right\}$, $\suite{\mathfrak A}$ is a sequence of renovating 
events of length 1 for the sequence $\suite{W^Z}$ since $\mathfrak A_n \subseteq \left\{W^Z_n=0\right\}$ for any $n$ (see \cite {BacBre02}, p.115, \cite{MR52:12118}, \cite{Foss92}). Moreover, this sequence is stationary in the sense that for all $n \ge 0$, $\mathfrak A_n=\theta^{-n}\mathfrak A_0,$ where $\mathfrak A_0= \left\{Y_{\sigma+D,\xi}=0\right\}.$ 
Hence since (\ref{eq:condsuffGGd}) amounts to $\prp{\mathfrak A_0}>0$, this is from \cite{BacBre02}, Theorem 2.5.3., a sufficient condition for the existence of a solution $W$ to (\ref{eq:recurstatFIFO}), and for strong backwards coupling to occur for $\suite{W^Z}$ with $W$. 

\textit{ }\\
\emph{Uniqueness}. Let $W$ be a solution of (\ref{eq:recurstatFIFO}). 
Then, we have $\prp{W\le D}>0$. Indeed, if $W > D$, $\pp$-a.s. (which implies in particular that $W\circ\theta>0$, $\pp$-a.s.), then  $W\circ\theta=W-\xi,$ $\pp$-a.s., which is absurd in view of the Ergodic Lemma (\cite{BacBre02}, Lemma 2.2.1). 
On another hand, the inequality (\ref{eq:ineq}) imply that on $\left\{W \le Y_{\sigma+D,\xi}\right\}$, 
\begin{equation*}
W\circ\theta\le \left[W\vee(\sigma+D)-\xi\right]^+\le \left[Y_{\sigma+D,\xi}\vee(\sigma+D)-\xi\right]^+=Y_{\sigma+D,\xi}\circ\theta.
\end{equation*}  
Thus $\left\{W \le Y_{\sigma+D,\xi}\right\}$ is $\theta$-contracting, whereas 
on $\{W\le D\}$, 
$$W\circ\theta=\left[(W+\sigma)\car_{W\le D}-\xi\right]^+\le \left[(D+\sigma)\car_{W\le D}-\xi\right]^+\le Y_{\sigma+D,\xi}\circ\theta,$$
hence $\left\{W \le Y_{\sigma+D,\xi}\right\}$ is $\pp$-almost sure. Therefore,  
$\suite{W^W}=\suiten{W\circ\theta^n}$ admits $\suite{\mathfrak A}$ as a stationary sequence of renovating events of length 1. From \cite{BacBre02}, Remark 2.5.3, $\prp{\mathfrak A_0}>0$ implies the uniqueness property. 
Finally, for any $x\in\R+$, we have $\pp$-a.s. that 
\begin{multline*}
x\vee(D\wedge \sigma)=(D\wedge \sigma)\car_{x\le D\wedge \sigma}+x\car_{x>D\wedge \sigma}\car_{x\le D}+x\car_{x>D}\\
\le (x+\sigma)\car_{x\le D\wedge \sigma}\car_{x\le D}+(x+\sigma)\car_{x>D\wedge \sigma}\car_{x\le D}+x\car_{x>D}=x+\sigma\car_{x\le D}.
\end{multline*}
This implies that on $\left\{Y_{\sigma\wedge D,\xi}\le W\right\}$,
$$Y_{\sigma\wedge D,\xi}\circ\theta=\left[Y_{\sigma\wedge D,\xi}\vee(\sigma\wedge D)-\xi\right]^+\le \left[W\vee(\sigma\wedge D)-\xi\right]^+\le 
W\circ\theta,$$ 
thus $\left\{Y_{\sigma\wedge D,\xi} \le W\right\}$ is $\theta$-contracting. 
It is $\pp$-almost sure since it includes $\left\{Y_{\sigma+D,\xi}=0\right\}$ (in view of the immediate fact that $Y_{\sigma\wedge D,\xi}\le Y_{\sigma+D,\xi}$, $\pp$-a.s.). 
\end{proof}

\subsection{Some applications}
\label{subsubsec:appli}
It is a very classical statement, that due to the FIFO discipline, the construction of the stationary versions of some quantities of interest can be derived from that of the workload sequence. In particular, provided that (\ref{eq:condsuffGGd}) holds, one can construct a congestion process and a departure process that are jointly compatible with the arrival process $\procz{N}$. Let us remark, that under condition (\ref{eq:condsuffGGd}) there exists also a stationary \emph{loss probability}, denoted $\pi(b)$, which is the 
probability that the waiting time proposed to a customer exceeds his initial patience, at equilibrium. This reads 
\begin{equation}
\label{eq:defpi}
\pi(b)=\prp{W>D}.
\end{equation}
With Theorem \ref{thm:borov} in hand, we have in particular that 
\begin{equation}
\label{eq:encadrepiB}
\prp{Y_{\sigma\wedge D,\xi}>D}\le \pi(b) \le \prp{Y_{\sigma+D,\xi}>D}.
\end{equation}

\subsection{Weak stationarity}
\label{subsubsec:anan}
In this section, condition (\ref{eq:condsuffGGd}) is no longer assumed to hold. We show how the techniques developed in \cite{Anan97} may allow us to 
construct a stationary workload for the queue, on an enriched probability space. Again, this is done using the stochastic comparison with the LRMST 
sequence (see (\ref{eq:ineq})).  
Let $\phi\left\{.\right\}$ be the measurable random map from $\left(\R,\B(\R)\right)$ into itself (we denote $\phi \in \mathfrak M(\R)$) defined 
by $$\phi(\omega)\left\{y\right\}=\left[y+\sigma(\omega)\car_{\left\{y\le D(\omega)\right\}}-\xi(\omega)\right]^+.$$
We work on the enlarged probability space $\Omega \times \R$, on which we define the shift 
$$\tilde\theta(\omega,x)=\left(\theta\omega,\phi(\omega)\{x\}\right).$$
We then have the following result. 
\begin{theorem}
\label{thm:stabWB}
The stochastic recursion (\ref{eq:recurstatFIFO}) admits a \emph{weak solution}, 
that is, a $\tilde \theta$-invariant probability $\qp$ on 
$\Omega \times \R$ whose $\Omega$-marginal
is $\pp$. Therefore, on 
$\left(\Omega \times \R\right)$  
there exists a $\R\times \mathfrak M(\R)$-valued r.v. $\left(\tilde W,\tilde \phi\right)$ 
satisfying
$$\tilde W\circ\tilde\theta=\tilde \phi\left\{\tilde W\right\}.$$ 
In particular, $\suiten{\left(\tilde W,\tilde \phi\right)\circ \tilde \theta^n}$ is stationary under $\qp$, and $\suiten{\tilde\phi\circ\tilde\theta^n}$ has the same distribution as $\suiten{\phi\circ\theta^n}$.  
\end{theorem}

\begin{proof}
We aim to apply Theorem 1 of \cite{Anan97}, whose corrected version is presented in \cite{Anan99}. Let us check that its hypotheses are met in our case. 
First, the sequence $\suite{L^0}$ is tight since it converges weakly, which implies using (\ref{eq:ineq}) and an immediate induction 
that $\suite{W^0}$ is tight, since for all $\varepsilon>0$, there exists $M_{\varepsilon}$ such that for all 
$n \in \N$, 
$$\prp{W^0_n \le M_{\varepsilon}}\ge\prp{L^0_n \le M_{\varepsilon}}\ge 1-\varepsilon.$$ 
Define now on $\Omega\times\R$ the random variables   
$$\tilde W(\omega,x):=x,\,\,\,\tilde \phi\left(\omega,x\right):=\phi(\omega),$$
and for all $n\in\N$, 
$$\tilde W_n(\omega,x):=\tilde W\left(\tilde \theta^n(\omega,x)\right).$$ 
Remark that for all $n\in\N$, $\mathfrak A \in \mathcal F$ and $\mathfrak B \in \mathcal B(\R)$,
$$\pp\otimes\delta_0\left[\tilde \theta^{-n} \left(\mathfrak A \times \R\right)\right]=\pp\otimes\delta_0\left[\mathfrak A \times \R\right]=\prp{\mathfrak A}$$
and  
\begin{multline*}
\pp\otimes\delta_0\left[\tilde \theta^{-n} \left(\Omega \times \mathfrak B\right)\right]=\pp\otimes\delta_0\left[\tilde \theta^{-n} \left(\tilde W_0^{-1}(\mathfrak B)\right)\right]\\=\pp\otimes\delta_0\left[\tilde  W_n^{-1}(\mathfrak B)\right]=\prp{W^0_n \in \mathfrak B}.
\end{multline*}
Hence, the probability distributions $\suiten{\left(\pp\otimes\delta_0\right)\circ \tilde \theta^{-n}}$ on $\Omega\times\R$ have $\Omega$-marginal $\pp$ and $\R$-marginals the distributions of $\suite{W^0}$, which form a tight sequence.   
The sequence $\suiten{\left(\pp\otimes\delta_0\right)\circ \tilde \theta^{-n}}$ is thus tight. 
On another hand, let us define for all $p \in \N^*$,
\begin{itemize}
\item[(i)] $\mathcal V_p=\left\{(\omega,x) \in \Omega\times\R;\,\,D(\omega)<x<D(\omega)+2^{-p}\right\},$
\item[(ii)] for any  $(\omega,x) \in \Omega\times\R$, $$f_p(\omega,x)=\car_{x\le D(\omega)}+\left(-2^{p}x+1+2^pD(\omega)\right)\car_{(\omega,x) \in \mathcal V_p},$$
\item[(iii)] for any $(\omega,x) \in \Omega\times\R$, $$\tilde \theta_p(\omega,x)=\left(\theta\omega,\left[x+f_p(\omega,x)\sigma(\omega)-\xi(\omega)\right]^+\right).$$
\end{itemize}
It is then easily checked, that for all $p$, $\mathcal V_p$ is an open set, $\tilde \theta=\tilde \theta_p$ outside $\mathcal V_p$, and that $\tilde \theta _p$ is continuous from $\omega\times\R$ into $\R$. 
Let us now fix $n,p \ge 1$. We have  
\begin{equation}
\label{eq:cesaro}
\frac{1}{n}\sum_{i=0}^{n-1} \left(\pp\otimes\delta_0\right)\circ \tilde \theta^{-i}\left(\mathcal V_p\right)=\frac{1}{n}\sum_{i=1}^{n}
 \prp{W^0_{i}\circ\theta^{-i} \in \left(D,D+2^{-p}\right)},
\end{equation}
where in words, $W^0_i\circ\theta^{-i}$ is the workload in the system at time 0 assuming that $C_{-i}$ finds an empty system upon arrival.   
Let for any $i\ge 1$, $A_i(\omega)$ denote the set of indices of the customers present in the system at time $0$ given $C_{-i}$ enters an empty system, and for any $j \in A_i$, $R_j$ be the remaining service time (i.e. service required minus service already completed) of 
$C_j$. Remark that $W^0_i\circ\theta^{-i}=\sum_{j\in A_i} R_j$. Denote as well 
$$\tau(\omega)=\sup\biggl\{j \ge 1; D_{-j}(\omega)+\sigma_{-j}(\omega)-\sum_{k=1}^j\xi_{-j}(\omega)>0\biggl\}.$$ 
Since for any $j\in \Z$, $C_j$ does not remain in the system 
for a time larger than $\sigma_{j}+D_j$, we have that $\bigcup_{i\ge 1}A_i(\omega) \subseteq [1,\tau(\omega)]$, $\pp$-a.s.. 
Hence, for any $i\ge 1$,    
\begin{multline*}
\prp{W^0_{i}\circ\theta^{-i} \in \left(D,D+2^{-p}\right)}=\prp{\sum_{j\in A_i} R_j\in \left(D,D+2^{-p}\right)}\\
\le \prp{\tau>i}+\prp{\left\{\tau \le i\right\}\bigcap\left\{\sum_{j\in A_i\cap [1,\tau]}R_j\in \left(D,D+2^{-p}\right)\right\}}\\
\le \prp{\tau>i}+\prp{\left\{\tau <\infty\right\}\bigcap\Biggl\{\bigcup_{k \le \tau}\bigcup_{j_1\le j_2 \le...\le j_k}\left\{\sum_{l=1}^kR_{j_l} \in \left(D,D+2^{-p}\right)\right\}\Biggl\}}. 
\end{multline*} 
From the Ergodic Theorem, the sequence $\{D_{-j}+\sigma_{-j}-\sum_{k=1}^j\xi_{-j}\}_{j\ge 1}$ tends $\pp$-a.s. to 
$-\infty$, hence $\tau<+\infty$, $\pp$-a.s., so that $\prp{\tau>i}$ tends to zero as $i$ goes to infinity.  
Therefore in view of (\ref{eq:cesaro}), we have that 
\begin{multline*}
\underset{p \rightarrow \infty}{\lim}\underset{n \rightarrow \infty}{\lim\inf}\frac{1}{n}\sum_{i=0}^{n-1} \left(\pp\otimes\delta_0\right)\circ \tilde \theta^{-i}\left(\mathcal V_p\right)\\
\le \underset{p \rightarrow \infty}{\lim}\prp{\left\{\tau <\infty\right\}\bigcap\Biggl\{\bigcup_{k \le \tau}\bigcup_{j_1\le j_2 \le...\le j_k}\left\{\left(\sum_{l=1}^kR_{j_l}-D\right) \in \left(0,2^{-p}\right)\right\}\Biggl\}}\\=0,
\end{multline*}
%
thus the last assumption of Theorem 1 of \cite{Anan97} is verified (see again \cite{Anan99} for the additional assumption (A3), p.272).  
We can therefore apply this result, yielding that there exists a $\tilde \theta$-invariant probability $\qp$ on 
$\Omega \times \R$ whose $\Omega$-marginal
is $\pp$. 
It is now straightforward that
$$\tilde W_n(\omega,x)=\phi\left(\theta^{n-1}\omega\right)\circ\phi\left(\theta^{n-2}\omega\right)\circ...\circ\phi(\omega)\{x\},$$
$$\tilde \phi\circ \tilde \theta^n(\omega,x)=\phi\circ\theta^n(\omega),$$
hence the sequence $\suite{\tilde W}$ satisfies on 
$\Omega \times \R$ the stochastic recursion
\begin{equation*}
\tilde W_{n+1}=\tilde \phi\circ\tilde\theta^n\left\{\tilde W_n\right\},
\end{equation*}
where $\suiten{\tilde W_n,\tilde \phi\circ\tilde\theta^n}$ is stationary under $\qp$. 
\end{proof}

\subsection{The loss system G/G/1/1}
\label{subsubsec:flipo}
The stationary workload of the loss system G/G/1/1 queue (there is no buffer, so that each customer is served if and only if he finds an empty system upon arrival), which is constructed in \cite{BacBre02}, section 2.6, and on an enlarged probability space in \cite{Fli83} and \cite{Anan97} is a particular case of the one constructed here, setting $$D(\omega)=0,\,\pp\mbox{-a.s.}.$$
Taking $D$ to be identically zero in the whole section \ref{sec:FIFOIB} yields the results mentioned above.

\section{Impatience until the end of service}
\label{sec:FIFOIE}   
Let us now consider a G/G/s/s+G(e) queue~: the model is that of the 
previous sections, except that the customers are now assumed to remain 
impatient until the end of their service. Indeed, they leave 
the system, and are eliminated forever, provided that their service is not 
\emph{completed} before their deadline. Keeping the notations and other assumptions 
of the previous section, customer $C_n$ is thus discarded when  
the total time he has to wait in the buffer and spend in the service booth is larger than his 
initial time credit $D_n$. We assume that the customers are unaware of their waiting time and deadline, and consequently wait in the system, and possibly enter service, as long as their deadline is not reached. 
In other words, for any $n\in\Z$ the maximal sojourn time of $C_n$ in the system is given by $D_n$, whereas its minimal sojourn 
time is $\sigma_n\wedge D_n$. 
Then, the LRmST sequence is that of the system with impatience until the beginning of service, whereas the LRMST sequence $\suite{L}$ is driven 
by the recursive equation 
$$L_{n+1}=\left[L_n\vee D_n-\xi_n\right]^+,$$ 
and in view of Lemma \ref{lemma:solrecur}, the unique stationary LRMST reads 
\begin{equation}
\label{eq:defQ}
Y_{D,\xi}=\left[\sup_{j \in \N^*} \left(D_{-j}-\sum_{i=1}^{j}\xi_{-i}\right)\right]^+.
\end{equation}
Similar arguments as those of section \ref{sec:regenerIB} lead to the following result. 
\begin{theorem}
\label{thm:regenerGGe}
The G/G/s/s+G(e) queue is regenerative if  
\begin{equation}
\label{eq:condsuffGGe}
\prp{Y_{D,\xi}=0}>0,
\end{equation}
and only if (\ref{eq:condnecGGd}) holds. 
\end{theorem}


The patience of some 
customer may finish while he is in service. Hence such customers participate to the workload, since some service is provided to them, whereas 
their service is not eventually completed. More precisely, the quantity of work added to the workload $W_n$ at the arrival of customer 
$C_n$ is given by
\[\left\{\begin{array}{ll}
\sigma_n &\mbox{ if $W_n \le (D_n-\sigma_n)^+$},\\
\sigma_n-(W_n+\sigma_n-D_n)=D_n-W_n &\mbox{ if $(D_n-\sigma_n)^+<W_n\le D_n$},\\
0 &\mbox{ if $W_n>D_n$.}
\end{array}\right.\]
This can be reformulate in a compact form, stating that the workload sequence is driven by the recursive equation 
$$W_{n+1}=\left[W_n+\left(\sigma_n-\left(W_n+\sigma_n-D_n\right)^+\right)^+-\xi_n\right]^+.$$ 
Therefore, a stationary workload for this queue is a $\R+$-valued r.v. $S$ that solves the equation
\begin{equation}
\label{eq:recurstatend}
S\circ\theta=\left[S+\left(\sigma-\left(S+\sigma-D\right)^+\right)^+-\xi\right]^+.
\end{equation}
We have the following result. 
\begin{theorem}
\label{thm:borovend}
\emph{(i)} The equation (\ref{eq:recurstatend}) admits a finite solution $S$ that is such that $Y_{\sigma\wedge D,\xi}\le S \le Y_{D,\xi}$.\\ 
\emph{(ii)} Provided that (\ref{eq:condsuffGGe}) holds, this solution 
is unique and for any r.v. $Z$ such that $Z\le Y_{D,\xi}$, $\pp$-a.s., $\suite{W^Z}$ converges with strong backwards coupling to $S$.\\ 
\emph{(iii)} If in addition (\ref{eq:condsuffGGd}) holds, then the unique solution is such that $S\le W$, $\pp$-a.s., where $W$ is the only solution of (\ref{eq:recurstatFIFO}).
\end{theorem}

\begin{proof} 
\emph{(i)} The mapping $x\mapsto \left[S+\left(\sigma-\left(S+\sigma-D\right)^+\right)^+-\xi\right]^+$ is $\pp$-a.s. non-decreasing and continuous, 
as easily 
checked. Hence a minimal solution $S$ to (\ref{eq:recurstatend}) can be constructed using Loynes' Theorem. 
Let us now remark that for any $x$, $\pp$-a.s., 
\begin{multline}
\label{eq:ineq2}
x+\left(\sigma-\left(x+\sigma-D\right)^+\right)^+=\left((x+\sigma)\wedge D\right)\car_{x \le D}+x\car_{x > D}\\
\le (x\vee D)\wedge\left(x+\sigma\car_{x\le D}\right).
\end{multline}
This clearly implies that the event $\{S\le Y_{D,\xi}\}$ is $\theta$-contracting. On another hand, $S$ is such that $\prp{S\le D}> 0$, since the contrary would imply that $S\circ\theta=S-\xi$, $\pp$-a.s., a contradiction to the Ergodic Lemma. 
But on $\{S\le D\}$, 
$$S\circ\theta=\left[\left((S+\sigma)\wedge D\right)-\xi\right]^+\le \left[D\vee Y_{D,\xi}-\xi\right]^+=Y_{D,\xi}\circ\theta,$$ 
thus $S \le Y_{D,\xi}$, $\pp$-a.s. 
Now, for any $x$ we also have that $\pp$-a.s., 
\begin{multline}
\label{eq:ineq3}
x\vee(D\wedge \sigma)=(D\wedge \sigma)\car_{x\le D\wedge \sigma}\car_{x\le D}+x\car_{x> D\wedge \sigma}\car_{x\le D}+x\car_{x>D}\\
\le (D\wedge (x+\sigma))\car_{x\le D\wedge \sigma}\car_{x\le D}+(D\wedge (x+\sigma))\car_{x> D\wedge \sigma}\car_{x\le D}+x\car_{x>D}\\=x+\left(\sigma-(x+\sigma-D)^+\right)^+,
\end{multline}
which implies that $\{Y_{\sigma\wedge D,\xi} \le S\}$ is $\theta$-contracting. Assuming that $Y_{\sigma\wedge D,\xi}>\sigma\wedge D$, 
$\pp$-a.s. would again contradict the Ergodic Lemma. Thus $\{Y_{\sigma\wedge D,\xi} \le S\}$ is $\pp$-almost sure since 
on $\{Y_{\sigma\wedge D,\xi}\le \sigma\wedge D\}$, 
$$Y_{\sigma\wedge D,\xi}\circ\theta=\left[\sigma\wedge D-\xi\right]^+\le \left[\left((S+\sigma)\wedge D\right)\car_{S \le D}+S\car_{S > D}-\xi\right]^+=S\circ\theta.$$

\emph{(ii)} 
For any solution $S^{\prime}$ of (\ref{eq:recurstatend}), $\left\{S^{\prime} \le S\right\}$ is $\theta$-contracting. This event is thus  
$\pp$-almost sure whenever (\ref{eq:condsuffGGe}) holds since it is included in $\{Y_{D,\xi}=0\}$. Hence the uniqueness of the solution is 
entailed by the minimality of $S$. On another hand, the inequality (\ref{eq:ineq2}) implies in particular with a simple induction that 
$W_n^Z\le L_n^{Y_{D,\xi}}=Y_{D,\xi}\circ\theta^n$ for all $n\in\N$ whenever $Z\le Y_{D,\xi}$. 
Thus, for all r.v. $Z$ such that $Z\le Y_{D,\xi}$, $\pp$-a.s., $\left\{\{Y_{D,\xi}\circ\theta^n=0\}\right\}_{n\in\N}$ is a 
sequence of renovating events of length 1 for $\suite{W^0}$. The strong backwards coupling property then follows, as in the proof of 
Theorem \ref{thm:borov}.

\emph{(iii)}  
The fact that the mapping $x\mapsto \left[x+\left(\sigma-\left(x+\sigma-D\right)^+\right)^+-\xi\right]^+$ is $\pp$-a.s. non-decreasing 
imply together with (\ref{eq:ineq2}) that on $\left\{S\le W\right\}$, 
$$S\circ\theta \le \left[W+\left(\sigma-\left(W+\sigma-D\right)^+\right)^+-\xi\right]^+\le W\circ\theta.$$ 
Thus $\left\{S\le W\right\}$ is $\theta$-contracting. This event is $\pp$-almost sure whenever (\ref{eq:condsuffGGd}) holds since it includes $\left\{Y_{D+\sigma,\xi}=0\right\}$. 
\end{proof}

\emph{Applications.} 
It is then possible to define for this model a stationary loss probability $\pi(e)$, which is the probability that the 
patience of customer $C_0$ is less than the sum of the stationary workload and his service time, i.e.
$$\pi(e)=\prp{S>D-\sigma}.$$
From \emph{(i)} of Theorem \ref{thm:borovend}, we have that 
$$\prp{Y_{\sigma\wedge D,\xi}>D-\sigma} \le \pi(e) \le \prp{Y_{D,\xi}>D-\sigma}.$$
On another hand, the stationary probability $\hat{\pi}(e)$ that a customer does not reach the server is given by 
$$\hat{\pi}(e)=\prp{S>D}.$$ 
Then in view of (\ref{eq:defpi}) and \emph{(iii)} of Theorem \ref{thm:borovend}, the loss probability $\pi(b)$ of G/G/1/1+G(b) is larger than $\hat{\pi}(e)$ for the same parameters.


\end{document}